\documentclass{amsart}
\usepackage[utf8]{inputenc}
\usepackage{geometry}
\usepackage{times}
\usepackage{epsfig}
\usepackage{tikz}
\usepackage{graphicx}
\usepackage{amsmath}
\usepackage{amssymb}
\usepackage{comment}
\usepackage{dsfont}

\usepackage{url}
\allowdisplaybreaks

\newcommand{\redd}[1]{{\color{black} #1}}

\usepackage{amsmath, amsfonts, amsthm}
\usepackage{float}
\usepackage{mathtools}
\usepackage{xcolor}

\DeclarePairedDelimiterX\set[1]\lbrace\rbrace{#1}

\newtheorem{theorem}{Theorem}
\newtheorem{proposition}[theorem]{Proposition}
\newtheorem{lemma}[theorem]{Lemma}

\newtheorem{corollary}[theorem]{Corollary}
\newtheorem{remark}[theorem]{Remark}

\theoremstyle{definition}

\newtheorem{definition}[theorem]{Definition}

\author{Neil Pritchard}
\author{Thomas Weighill}
\address{University of North Carolina at Greensboro, Greensboro NC, USA}

\begin{document}

\title{Coarse embeddability of Wasserstein Space and the Space of Persistence Diagrams}

\begin{abstract}
    We prove an equivalence between open questions about the embeddability of the space of persistence diagrams and the space of probability distributions (i.e.~Wasserstein space).  It is known that for many natural metrics, no coarse embedding of either of these two spaces into Hilbert space exists. Some cases remain open, however. In particular, whether coarse embeddings exist with respect to the $p$-Wasserstein distance for $1\leq p\leq 2$ remains an open question for the space of persistence diagrams and for Wasserstein space on the plane. In this paper, we show that embeddability for persistence diagrams \redd{is equivalent to} embeddability for Wasserstein space on $\mathbb{R}^2$.  \redd{When $p > 1$,  Wasserstein space on $\mathbb{R}^2$ is snowflake universal (an obstruction to embeddability into any Banach space of non-trivial type) if and only if the space of persistence diagrams is snowflake universal.}
\end{abstract}

\maketitle

\section{Introduction}

In this paper, we consider embeddings of two kinds of non-linear data: persistence diagrams, and probability distributions. A persistence diagram is an unordered set of points in the plane which arises as a summary of the topological information in a dataset (e.g.~a point cloud or a grayscale image). Persistence diagrams have proven to capture important information in applications involving image data~\cite{chung2020smooth}, geospatial data~\cite{feng2021persistent}, time series data~\cite{seversky2016time}, and more. The set of persistence diagrams is endowed with a family of natural metrics called Wasserstein distances (see Section \ref{sec:prelim}). In practice, the analysis of persistence diagrams is hampered by the fact that the space of persistence diagrams is not readily identifiable with a subset of Euclidean space, which limits the use of classical machine learning and statistical techniques. Hence, many vectorizations (i.e.~maps from the set of persistence diagrams to a Hilbert space) have been introduced in recent years. Examples of vectorizations for persistence diagrams include persistence landscapes~\cite{bubenik2015statistical}, persistence images~\cite{adams2017persistence} and persistence curves~\cite{chung2022persistence}. Ideally, one would like these maps to be isometric embeddings; unfortunately, it can be shown theoretically that no isometric embedding of the space of persistence diagrams into Hilbert space exists. Even worse, even if one relaxes the isometric condition -- say, to that of a coarse embedding -- such an embedding is still theoretically proven not to exist in most cases (see Section \ref{sec:previous} for a survey of such results). 

This paper also studies probability distributions as objects. The space of all (Borel, finite moment) probability distributions on $\mathbb{R}^n$ is equipped with a family of metrics, also called Wasserstein distances; we refer to this space as Wasserstein space on $\mathbb{R}^n$. Optimal transport and Wasserstein distances between distributions have been applied in a variety of areas including economics, machine learning, computer graphics and fluid dynamics. Like persistence diagrams, probability distributions, with Wasserstein metrics, are also difficult to embed into Euclidean space (see Section~\ref{sec:previous}). 

Despite the large number of negative results regarding embeddings, some important cases remain open. In the case of persistence diagrams, it is not known whether the set of persistence diagrams with the $p$-Wasserstein metric coarsely embeds into Hilbert space for $1 \leq p \leq 2$. In the case of probability distributions, it is not known whether the set of probability distributions on $\mathbb{R}^2$ with the $p$-Wasserstein metric coarsely embeds into Hilbert space for the same range of $p$ values. These spaces are somewhat similar in that persistence diagrams can be thought of as discrete distributions, and one might expect that the answers to these two open questions should be related. In this paper, we confirm this by leveraging a result of Nowak on coarse embeddings of finite subsets~\cite{nowak2005coarse}. In particular, we show that all finite sets of distributions with the $p$-Wasserstein metric uniformly coarsely embed into the space of persistence diagrams with the $p$-Wasserstein metric (Proposition \ref{eqi}), and, \redd{conversely, that finite sets of persistence diagrams embed into Wasserstein space (Proposition~\ref{eqipd} and \ref{equipdbi}). As a corollary, we obtain our main result: that the space of persistence diagrams with the $p$-Wasserstein metric embeds into Hilbert space if and only if the same is true of Wasserstein space on $\mathbb{R}^2$ with the $p$-Wasserstein metric.}

\redd{An important obstruction to coarse embeddability into any Banach space of non-trivial type is snowflake universality, which was studied for Wasserstein space on $\mathbb{R}^3$ in \cite{andoni2018snowflake}. Whether Wasserstein space on $\mathbb{R}^2$ is snowflake universal is still an open question. Using a similar argument involving finite subsets, we prove that for $p > 1$, the space of persistence diagrams with the $p$-Wasserstein metric is $\frac{1}{p}$-snowflake universal if and only if the same is true of Wasserstein space on $\mathbb{R}^2$ with the $p$-Wasserstein metric.}



\section{Preliminaries}\label{sec:prelim}

\subsection{Wasserstein Space} 
We recall the following basic notions in optimal transport from \cite{peyre2019computational}.

\begin{definition}
Let $(X,d_X)$ be a complete separable metric space and let $\mathcal{P}_p(X)$ denote the space of all Borel probability measures on $X$ with finite $p$-th moments, for $1 \leq p < \infty$. The \textbf{$p$-Wasserstein distance} between $\alpha,\beta \in \mathcal{P}_p(X)$ is given by
\begin{align*}
    (W_p(\alpha, \beta))^p = \text{inf}_{\gamma \in \mathcal{U}(\alpha,\beta)}\int d(x,y)^pd\gamma(x,y),
\end{align*}
where $\mathcal{U}(\alpha,\beta)$ is the set of Borel probability measures on $X^2$ with marginals $\alpha$ and $\beta$. The metric space $(\mathcal{P}_p(X),W_p)$, which we will usually denote simply as $\mathcal{P}_p(X)$, is called the\textbf{ Wasserstein $p$-space} over $(X,d_X)$.
\end{definition}

While the $p$-Wasserstein distance is defined for general measures, in this paper we will be particularly focused on discrete measures with finite support. In this case there is an equivalent formulation of the distance which will be useful. Let $\alpha = \sum_{j=1}^ma_j\delta_{x_j}$ and $\beta = \sum_{j=1}^{n}b_j\delta_{y_j}$ be discrete measures. The above definition reduces to
\begin{align*}
    (W_p(\alpha, \beta))^p = \text{min}_{P\in U(a,b)}\sum_{i,j}d(x_i,x_j)^pP_{ij}.
\end{align*}
Here $U(a,b)$ is the set of $n\times m$ matrices such that $P\mathds{1}_m = a$ and $P^T\mathds{1}_n = b$. Moreover, if one additionally assumes that $\alpha$ and $\beta$ have rational coefficients then by rewriting $\alpha = \sum_{j=1}^N\frac{1}{N}\delta_{x_j'}$ and $\beta  = \sum_{j=1}^{N}\frac{1}{N}\delta_{y_j'}$ it follows from the Birkhoff-von Neumann theorem that the distance can be re-expressed as
\begin{align}\label{Distance}
    (W_p(\alpha,\beta))^p = \text{min}_{\sigma \in P(N)}\frac{1}{N}\sum_{j = 1}^N d(x_j',y_{\sigma(j)}')^p 
\end{align}
where $P(N)$ is the set of permutations on $\{1, \ldots, N\}$ \cite{peyre2019computational}. 


\subsection{The Space of Persistence Diagrams}
Persistence diagrams typically appear in topological data analysis as a way to store topological information from a sequence of complexes. 
\begin{definition}
Denote by $\mathbb{R}^2_{<}$ the set $\{(b,d)\in \mathbb{R}^2\mid b<d\}$. A \textbf{persistence diagram} is finite multiset $D \subseteq \mathbb{R}^2_{<}$.
\end{definition}

Note that other definitions exist in the literature which allow, for example, countable multisets, or which for notational convenience define a persistence diagram to include infinitely many copies of the diagonal (e.g.~\cite{otter2017roadmap}). We work with the above definition as it is the most convenient context for our results. Persistence diagrams are often also allowed to include points of the form $(b, \infty)$. In this paper we do not allow infinite values, as we want all our distances to be finite. In applications, infinite values are often removed or replaced for the same reason. We now define Wasserstein distances between persistence diagrams.

\begin{definition}
A \textbf{partial matching} between two persistence diagrams $D_1$ and $D_2$ is a triple $(B_1,B_2,f)$ where $B_1 \subseteq D_1$, $B_2\subseteq D_2$ and $f\colon B_1\to B_2$ is a bijection.
\end{definition}

\begin{definition}
    For $1\leq p<\infty$ The \textbf{$p$-cost} of a partial matching $(B_1,B_2,f)$ between diagrams $D_1$ and $D_2$ is defined as,
\begin{align*}
    \text{cost}_p(f) &= \left(\sum_{(b,d) \in B_1} d((b,d),f((b,d)))^p\right. \vphantom) + 
    \sum_{(b,d)\in D_1\setminus B_1} d((b,d),\Delta_{\mathbb{R}})^p\\
    &+\left. \vphantom( \sum_{(b,d)\in D_2\setminus B_2} d((b,d),\Delta_{\mathbb{R}})^p\right)^{\frac{1}{p}}
\end{align*}
where $\Delta_{\mathbb{R}} = \{(x,x) \in \mathbb{R}^2\}$ is the diagonal. 
\end{definition}

The distance $d$ in the above definition is a distance between points in the plane. Common choices are an $\ell^q$ distance, for $q\geq 1$, or $\ell^\infty$ distance. Since all $\ell^q$ distances on $\mathbb{R}^2$ are bi-Lipschitz equivalent, our results do not depend on the particular choice of $d$; when necessary we assume the $\ell^\infty$ distance.

Now, for two persistence diagrams $D_1,D_2$ we define the distance function 
\begin{align*}
    W_p(D_1,D_2) = \inf\{\text{cost}_p(f)\mid \text{$f$ is a partial matching between $D_1$ and $D_2$}\}.
\end{align*}
Note that  $W_p$ is a metric, which we call the $p$-\textbf{Wasserstein metric on persistence diagrams}. When $p = \infty$, one can make the natural replacements of sums by suprema to obtain the bottleneck distance $W_\infty$.

\begin{definition}
    Let $\mathcal{D}$ denote the collection of persistence diagrams. The metric space $(\mathcal{D},W_p)$ is called the \textbf{space of persistence diagrams with the $p$-Wasserstein distance}. 
\end{definition}

 We use $W_p$ to denote both the Wasserstein distance between probability distributions and the Wasserstein distance between persistence diagrams; the meaning will always be clear from the arguments.  We adopt the convention that $\mathbb{R}^2$ is always quipped with the $\ell^\infty$ distance; other $\ell^q$ distances are bi-Lipschitz equivalent so we do not lose any generality with this choice. 

 \begin{remark}
     To each persistence diagram $D$ we can naturally associate the corresponding empirical distribution in the plane, i.e.~the uniform distribution on the points (possibly with multiplicity) in $D$. Note that the Wasserstein distance between two persistence diagrams is not the same as the Wasserstein distance between the corresponding empirical distributions as a result of partial matchings and the use of the diagonal as a universal point to match to. 
 \end{remark}

\subsection{Embeddings}
A metric embedding is a map between metric spaces which preserves distances in some sense. We will consider a number of different types of embeddings. 

\begin{definition}
    Let $(X,d_X)$ and $(Y,d_Y)$ be metric spaces and let $f\colon X\to Y$ be a map. We say that $f$ is
    \begin{itemize}
        \item an $\textbf{isometric embedding}$ if $\forall_{x,y \in X} \ d(x,y) = d(f(x),f(y))$;
        \item a \textbf{bi-Lipschitz embedding} if there exists a constant $A\geq 1$ such that 
        $$\forall_{x,y \in X}\ A^{-1}d(x,y)\leq d(f(x),f(y))\leq Ad(x,y),$$
        \item \textbf{$\epsilon$-quasi-isometric embedding} if there exists an $\epsilon>0$ such that 
        $$\forall_{x,y \in X}\ d(x,y)-\epsilon \leq d(f(x),f(y))\leq d(x,y)+ \epsilon, $$
        \item a \textbf{coarse embedding} if there exist non-decreasing functions $\rho_1,\rho_2\colon [0,\infty)\to [0,\infty)$ satisfying $\lim_{t\to \infty} \rho_1(t)= +\infty$ and 
        $$ \forall_{x,y \in X} \ \rho_1(d(x,y)) \leq d(f(x),f(y)) \leq \rho_2(d(x,y))$$
    \end{itemize}  
\end{definition}

Suppose now that $D\in [1,\infty)$, then $f$ is said to have \textbf{distortion at most $D$} if there exists an $s\in (0,\infty)$ such that 

$$\forall_{x,y \in X} \ \quad sd(x,y)\leq d(f(x),f(y))\leq Dsd(x,y).$$

Moreover, the distortion of $f$, denoted $\textbf{dist}(f)$, is the infimum over all $D$ such that the inequality above holds. If a map $f\colon X\to Y$ with distortion $D$ exists, then $X$ is said to embed into $Y$ with distortion $D$. We will use the notation 
$$c_Y(X) = \inf_{f\colon X\to Y} \text{dist}(f).$$

If $\theta \in (0,1]$, then the \textbf{$\theta$-snowflake} of a metric space $(X,d)$ is the metric space $(X, d^\theta)$, where the metric is the obtained by raising $d$ to the $\theta$ power. Following \cite{andoni2018snowflake}, a metric space $X$ is said to be \textbf{$\theta$-snowflake universal} if for every finite metric space $(Y,d_Y)$, $c_X(Y,d_Y^{\theta}) = 1$.

\section{Past results on embeddings}\label{sec:previous}

In this section we will give an overview of some embedding results for $p$-Wasserstein space and the space of persistence diagrams. In general, most of these spaces do not embed into Hilbert or Euclidean spaces except under severe restrictions. 

\begin{theorem}[Turner et al., 2014~\cite{turner2014frechet}]\label{thmiso}
$(\mathcal{D},W_p)$ does not admit an isometric embedding into Hilbert space for any $1\leq p\leq\infty$
\end{theorem}

 \begin{theorem}[Carri{\`e}re and Bauer, 2018~\cite{carriere2018metric}]\label{thmbi}
 Let $n\in \mathbb{N}$. Then for any $N\in \mathbb{N}$ and $L>0$, $(\mathcal{D}_N^L,W_p)$ does not bi-Lipschitz embed into $\mathbb{R}^n$ for $p\in \mathbb{N}\cup {\infty}$
 \end{theorem}

Here the $\mathcal{D}_N^L$ denotes a restricted space of persistence diagrams; namely the space of all diagrams which have at most $N$ points and whose points all lie in the region $[-L,L]^2$. To obtain a positive result, we need not only a cardinality restriction, but a relaxation of the embedding type.

\begin{theorem}[Mitra and Virk, 2018~\cite{mitra2021space}]
$(\mathcal{D}_N,W_p)$ coarsely embeds into Hilbert space for $1\leq p\leq\infty$, where $\mathcal{D}_N$ denotes the space of persistence diagrams with at most $N$ points.
\end{theorem}

Without the cardinality restriction, even coarse embeddability is not possible in many natural cases. A first result in this direction showed that the space of all persistence diagrams fails to have Yu's Property A~\cite{yu2000coarse}, a sufficient but not necessary condition for coarse embeddability in Hilbert space originally introduced in Yu's work on the coarse Baum-Connes and Novikov Conjectures.

\begin{theorem}[Bell et al., 2019~\cite{bell2021space}]
$(\mathcal{D},W_p)$ does not have Yu's Property $A$ for $1 \leq p<\infty$.
\end{theorem}

Later it was shown for all $p>2$ that the space of persistence diagrams fails to coarsely embed into Hilbert space. 
\begin{theorem}[Bubenik and Wagner, 2020~\cite{bubenik2020embeddings}]\label{waginf}
$(\mathcal{D},W_{\infty})$ does not coarsely embed into Hilbert space.
\end{theorem}

\begin{theorem}[Wagner, 2021~\cite{wagner2021nonembeddability}]\label{wagp}
$(\mathcal{D},W_p)$ does not coarsely embed into Hilbert space for $2<p<\infty$.
\end{theorem}

Turning to Wasserstein space, Andoni, Naor, and Neiman show that $p$-Wasserstein space on $\mathbb{R}^3$ is $\frac{1}{p}$-snowflake universal. The proof of this result relies on an explicit embedding of any snowflake of a finite metric space into $p$-Wasserstein space on $\mathbb{R}^3$ as a uniform measure.

\begin{theorem}[Andoni-Naor-Neiman, 2018~\cite{andoni2018snowflake}]
If $p \in (1,\infty)$ then for every finite metric space $(X,d_X)$ we have
\begin{align*}
    c_{\mathcal{P}_p(\mathbb{R}^3)}\left(X,d_X^{\frac{1}{p}}\right) = 1.
\end{align*}
\end{theorem}

As a corollary, Andoni-Naor-Neiman prove that $p$-Wasserstein space on $\mathbb{R}^3$ fails to coarsely embed into Hilbert space for $p > 1$. The case for $\mathbb{R}^2$ and $p = 1$ remains open.

\begin{theorem}[Andoni-Naor-Neiman, 2018~\cite{andoni2018snowflake}]
If $p>1$ then $\mathcal{P}_p(\mathbb{R}^3)$ does not admit a coarse embedding into any Banach space of nontrivial type. In particular, for $p>1$, $\mathcal{P}_p(\mathbb{R}^3)$ does not admit a coarse embedding into Hilbert space.
\end{theorem}

To summarize, coarse embeddability in Hilbert space remains an open question for persistence diagrams when $1 \leq p \leq 2$ and for Wasserstein space when the underlying space is the plane (the most closely related context to that of persistence diagrams). Our main results connect these two open questions.

\section{\redd{Embeddings of finite subsets}}

The following characterization, by Nowak, of coarse embeddabilty into Hilbert space by way of finite subsets will prove quite useful. 
\begin{theorem}[Nowak, 2005~\cite{nowak2005coarse}]\label{nowak}
A metric space $X$ admits a coarse embedding into a Hilbert space if and only if there exists non-decreasing functions $\rho_{-},\rho_{+}\colon [0,\infty) \to [0,\infty)$ such that $\lim_{t\to \infty} \rho_{-}(t) = \infty$ and for every finite subset $A \subset X$ there exists a map $f_A\colon A\to \ell_2$ satisfying, 
\begin{align*}
    \rho_{-}(d(x,y))\leq \|f_A(x)-f_A(y)\| \leq \rho_{+}(d(x,y))
\end{align*}
for every $x,y \in A$.
\end{theorem}

Nowak's characterization of coarse embeddings into Hilbert space allows one to restrict one's attention to maps on finite subsets. In light of this the proofs of Theorems $\ref{A}$ and $\ref{B}$ rely on finding low distortion maps on finite subsets of the relevant spaces. 

\subsection{ \redd{$\varepsilon$-Quasi-isometric embeddings for finite sets of distributions}}

We now proceed to show the existence of a suitable map from any finite subset of measures in $\mathcal{P}_p(\mathbb{R}^2)$, the $p$-Wasserstein space over $\mathbb{R}^2$, into the space of persistence diagrams. For convenience, we will use the notation $d(\cdot, \cdot)$ for distances in both $(\mathcal{D}, W_p)$ and $\mathcal{P}_p(\mathbb{R}^2)$ for the remainder of this section. By a \textbf{discrete rational measure} we mean a discrete measure $\alpha = \sum_{i=1}^n a_i \delta_{x_i}$ with finite support where all the $a_i$ are rational numbers. 

\begin{lemma}\label{isom}
Suppose $A = \{\alpha_1,\cdots, \alpha_n\}$ is a finite set of discrete rational measures in $\mathcal{P}_p(\mathbb{R}^2)$. Then there exists an isometric embedding $f\colon A \to (\mathcal{D}, W_p).$
\end{lemma}

\begin{proof}
Let $N$ denote the common denominator of all coefficients in the measures $\alpha_1,\cdots, \alpha_n$. We write each $\alpha_i$ as a sum of uniformly weighted Dirac measures (possibly with duplication):
\begin{align*}
\alpha_i = \sum_{j=1}^N\frac{1}{N}\delta_{x_j^i}.
\end{align*}
Let $D$ denote the diameter of the set $\{\frac{1}{N}x_j^i\}_{i,j}$ in $\mathbb{R}^2$. Note that there exists an $x \in \mathbb{R}^2$ such that 
\begin{align*}
    \{\frac{1}{N}x_j^i+x\} \subset \{(a,b) \in \mathbb{R}^2 \mid a<b\} \text{and}\\
    d(\{\frac{1}{N}x_j^i+x\}, \Delta_{\mathbb{R}}) > 2D.
\end{align*}

Define $f\colon A \to \mathcal{D}$ by $\alpha_i \mapsto \{\frac{1}{N}x_j^i+x\}$. From \eqref{Distance} we have that
\begin{align*}
    (W_p(\alpha,\beta))^p = \min_{\sigma \in P(N)}\frac{1}{N}\sum_{j = 1}^N d(x_j', y_{\sigma(j)}')^p.
\end{align*}
Now, since the diagrams are sufficiently far from the diagonal, we must have that the distance between the diagrams $f(\alpha)$ and $f(\beta)$ is achieved by a perfect matching, which proves the result.
\end{proof}

\begin{proposition}\label{eqi}
Suppose $A = \{\alpha_1,\cdots, \alpha_n\} \subseteq \mathcal{P}_p(\mathbb{R}^2)$. Then for all $\epsilon>0$ there exists an $\epsilon$-quasi-isometric embedding $f\colon A \to (\mathcal{D},W_p)$.
\end{proposition}

\begin{proof}
Let $A = \{\alpha_1,\cdots, \alpha_n\} \subseteq \mathcal{P}_p(\mathbb{R}^2)$ and $\epsilon>0$ be given. It follows from standard results in optimal transport~\cite{villani} that there exists a collection $B = \{\beta_1,\cdots, \beta_n\} \subseteq \mathcal{P}_p(\mathbb{R}^2)$ of rational discrete measures such that $\forall_i d(\alpha_i,\beta_i)<\frac{\epsilon}{2}$.
By Lemma \ref{isom} there exists an isometry $f\colon B \to \mathcal{D}$. Define $\bar{f}\colon A \to \mathcal{D}$ by $\alpha_i \mapsto f(\beta_i)$. It is easy to check that this defines an $\epsilon$-quasi-isometry.
\end{proof}

\subsection{\redd{bi-Lipschitz embeddings for finite sets persistence diagrams}}

We now consider the other direction, namely embedding persistence diagrams into Wasserstein space. This direction requires more care, with the main obstacle being the existence of matchings to the diagonal. \redd{We give two separate constructions: one which results in a bi-Lipschitz embedding, and one which results in a $\varepsilon$-quasi-isometry for arbitrarily small $\varepsilon$. The former, which we deal with in this subsection, is enough to derive the results for coarse embeddings into Hilbert space.The latter construction, given in the next subsection, only holds for $p > 1$ but is necessary for results regarding snowflake universality. }

\redd{

\begin{lemma}\label{PDembedbi}
Suppose $D = \{D_1,D_2,\cdots, D_n\} \subset (\mathcal{D},W_p)$ is a finite subset of diagrams. Then there exists a map $f\colon D \to \mathcal{P}_p(\mathbb{R}^2)$ satisfying the following:
\begin{align}\label{eq1}
    \frac{1}{N^{1/p}}d(D_i,D_j) \leq d(f(D_i),f(D_j)) \leq \frac{2^{1/p}}{N^{1/p}} d(D_i,D_j)
\end{align}
where $N$ is the total number of points in diagrams in $D$, counted with multiplicity.
\end{lemma}
\begin{proof}
We begin by fixing some notation. Let $\pi_{ij}$ denote the optimal partial matching between $D_i$ and $D_j$, letting $\mathcal{U}_{ij}$ denote those points in $D_i$ that are unmatched under this partial matching. For any $x \in \mathbb{R}^2$, let $\rho_x$ be the closest point on the diagonal to $x$.  For each $i$, define the following multiset of size $N$ (counted with multiplicity):
$$
\overline{D_i} = D_i \cup \bigcup_{j \neq i} \rho(D_j)
$$
where $\rho(D_j) = \{\rho_x \mid x \in D_j\}$. Define $f\colon D \to \mathcal{P}_p(\mathbb{R}^2)$ by sending each $D_i$ to the uniform measure on $\overline{D_i}$. Fix $i$ and $j$ and define a coupling $\pi$ from $f(D_i)$ to $f(D_j)$ as follows:
$$
\pi = \frac{1}{N}  \left[ \sum_{x \in D_i \setminus \mathcal{U}_{ij}} \left( \delta_{(x,\pi_{ij}(x))} + \delta_{(\rho_{\pi_{ij}(x)}, \rho_x)} \right)
+ \sum_{x \in \mathcal{U}_{ij}} \delta_{(x,\rho_x)}
+ \sum_{x \in \mathcal{U}_{ji}} \delta_{(\rho_x, x)}
+ \sum_{x \in \overline{D_i} \setminus (D_i \cup D_j)} \delta_{(x,x)}
  \right]
$$
Note that the mapping $x \mapsto \rho_x$ is distance non-increasing. Since this is a coupling, we have
\begin{align*}
& d(f(D_i), f(D_j))^p  \\
\leq \ & \frac{1}{N} \left[ \sum_{x \in D_i \setminus \mathcal{U}_{ij}} \left( \|x-\pi_{ij}(x)\|_{\infty}^p + \|\rho_{\pi_{ij}(x)}-\rho_x\|_{\infty}^p \right)
+ \sum_{x \in \mathcal{U}_{ij}} \|x-\rho_x\|_{\infty}^p
+ \sum_{x \in \mathcal{U}_{ji}} \|\rho_x - x\|_{\infty}^p
  \right] \\
  \leq \ & \frac{1}{N} \left[ \sum_{x \in D_i \setminus \mathcal{U}_{ij}} 2 \|x-\pi_{ij}(x)\|_{\infty}^p
+ \sum_{x \in \mathcal{U}_{ij}} \|x-\rho_x\|_{\infty}^p
+ \sum_{x \in \mathcal{U}_{ji}} \|\rho_x - x\|_{\infty}^p
  \right] \\
  \leq \ & \frac{2}{N} d(D_i, D_j)^p
\end{align*}

For the lower bound,  note that any bijective coupling of uniform measures $f(D_i)$ and $f(D_j)$ induces a partial matching on the diagrams $D_i$ and $D_j$. Let $\pi$ denote the optimal coupling between $f(D_i)$ and $f(D_j)$ and let $\tilde{\pi}_{ij}$ denote the induced partial matching between diagrams $D_i$ and $D_j$. Further let $\tilde{\mathcal{U}}_{ij}$ denote those points unmatched under $\tilde{\pi}_{ij}$, which must therefore be matched by $\pi$ to something on the diagonal. Then we have, 
\begin{align*}
    d(D_i,D_j)^p &\leq \sum_{x \in D_i\setminus \mathcal{U}_{ij}}\|x-\pi_{ij}(x)\|_{\infty}^p + \sum_{x \in \mathcal{U}_{ij}}\|x -\rho_x\|_{\infty}^p +\sum_{x \in \mathcal{U}_{ji}} \|x -\rho_x\|_{\infty}^p\\
    &\leq \sum_{x \in D_i\setminus \mathcal{U}_{ij}}\|x-\pi(x)\|_{\infty}^p + \sum_{x \in \mathcal{U}_{ij}}\|x -\pi(x)\|_{\infty}^p + \sum_{x \in \mathcal{U}_{ji}} \|x -\pi^{-1}(x)\|_{\infty}^p  \\
    & \leq N d(f(D_i),f(D_j))^p
\end{align*}
\end{proof}
}

\begin{figure}[h]
    \centering
    \begin{tikzpicture}[scale=2]
    \begin{scope}
        \draw[->] (0,0)--(0,2);
        \draw[->] (0,0)--(2,0);
        \draw[dashed] (0,0)--(2,2);
        \node[anchor=center] at (1,1.5) {\color{blue} $\blacktriangle$};
        \node[anchor=center] at (0.2,1.4) {\color{blue} $\blacktriangle$};
        \node[anchor=center] at (0.3,0.7) {\color{orange} $\bullet$};
        \node[anchor=center] at (0.5,1.9) {\color{orange} $\bullet$};
        \node[anchor=center] at (1.5,1.9) {\color{orange} $\bullet$};
        \draw[|->] (2.5,1)--(3,1); 
    \end{scope}
    \begin{scope}[xshift=100]
        \draw[->] (0,0)--(0,2);
        \draw[->] (0,0)--(2,0);
        \node[anchor=center] at (1,1.5) {\color{blue} $\blacktriangle$};
        \node[anchor=center] at (0.2,1.4) {\color{blue} $\blacktriangle$};
        \node[anchor=center] at (0.3,0.7) {\color{orange} $\bullet$};
        \node[anchor=center] at (0.5,1.9) {\color{orange} $\bullet$};
        \node[anchor=center] at (1.5,1.9) {\color{orange} $\bullet$};
        \node[anchor=center] at (1.2,1.2) {\color{blue} $\blacktriangle$};
        \node[anchor=center] at (0.5,0.5) {\color{blue} $\blacktriangle$};
        \node[anchor=center] at (1.7,1.7) {\color{blue} $\blacktriangle$};
        \node[anchor=center] at (1.25,1.25) {\color{orange} $\bullet$};
        \node[anchor=center] at (0.8,0.8) {\color{orange} $\bullet$};
    \end{scope}
    \end{tikzpicture}
    \caption{Sketch of the construction in the proof Lemma \ref{PDembed} for two diagrams $\{D_1, D_2\}$. The two persistence diagrams on the left (blue triangles and orange circles resp.) are sent by $f$ to the uniform distributions shown on the right, whose supports have the same cardinality.}
    \label{fig:proofbi}
\end{figure}
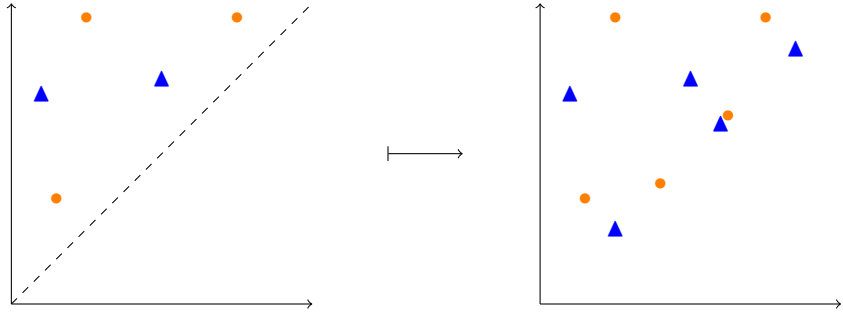

For a discrete measure $\alpha = \sum a_i\delta_{x_i}$ and a real number $r$ denote by $\lambda_r(\alpha)$ the dilated measure $\lambda_r(\alpha) = \sum a_i\delta_{rx_i}$. Note that by defining $\bar{f} = \lambda_{N^{1/p}} \circ f$ we may obtain a map on finite subsets of diagrams with the bounds contained in the following proposition.

\begin{proposition}\label{equipdbi}
Suppose $D = \{D_1,D_2,\cdots, D_n\} \subset (\mathcal{D},W_p)$ is a finite subset of diagrams. Then there exists an embedding $f\colon D \to \mathcal{P}_p(\mathbb{R}^2)$ such that  
$d(D_i,D_j) \leq d(f(D_i),f(D_j)) \leq 2^{1/p} d(D_i,D_j)$
\end{proposition}

\subsection{\redd{$\varepsilon$-Quasi isometric embeddings for finite sets of persistence diagrams ($p>1$)}}

Inspired by the explicit construction in \cite{andoni2018snowflake}, we now show how to construct an embedding of a finite subset of persistence diagrams into the space of discrete distributions by adding a \redd{sufficiently large number of} extra points along the diagonal. See Figure \ref{fig:proof} for an illustration of our construction.

\begin{lemma}\label{PDembed}
Suppose $D = \{D_1,D_2,\cdots, D_n\} \subset (\mathcal{D},W_p)$ is a finite subset of diagrams whose points all have multiplicity one. Let $p>1$, then for all $\epsilon >0$ there exists an $N \in \mathbb{N}$ and map $f\colon D \to \mathcal{P}_p(\mathbb{R}^2)$ satisfying the following for sufficiently large $s$:
\begin{align}\label{eq1}
    \frac{d(D_i,D_j)}{(N+s+1)^{1/p}} \leq d(f(D_i),f(D_j)) \leq \frac{d(D_i,D_j)}{(N+s+1)^{1/p}} + \frac{\epsilon}{(N+s+1)^{1/p}}
\end{align}
\end{lemma}

\begin{proof}

Let $D$ be as above and let $\epsilon >0$ be given. We will map each diagram to a subset of the plane, such that these subsets all have the same cardinality. \redd{We recall some notation from the proof of Lemma \ref{PDembedbi}: let $\pi_{ij}$ denote the optimal partial matching between $D_i$ and $D_j$, $\mathcal{U}_{ij}$ denote those points in $D_i$ that are unmatched under this partial matching, and $\rho_x$ denote the closest point on the diagonal to $x \in \mathbb{R}^2$.} Let $N_i = |D_i|$, and $N = \max_i N_i$. Furthermore, define the following:
\begin{align*}
  M &= \max_{(x,y) \in \cup_iD_i} \frac{x+y}{2}\\
  m &= \min_{(x,y) \in \cup_iD_i} \frac{x+y}{2}\\
  I &= \{(m+\frac{t(M-m)}{s},m+\frac{t(M-m)}{s})\}_{t=0}^s\\
  I_i &=\{(m+\frac{t(M-m)}{s},m+\frac{t(M-m)}{s})\}_{t = s+1}^{s+N-N_i}.
\end{align*} 

Here $s > N$ is taken so large so that 

$$N \cdot \left(\frac{N(M-m)}{s}\right)^p < \frac{\epsilon^p}{3}\quad \text{and} \quad \forall_{i,j}\ (s+1+N-N_j-|\mathcal{U}_{ij}|)\left(\frac{N(M-m)}{s}\right)^p< \frac{\epsilon^p}{3},$$

which can be achieved if $p > 1$. We note that $|I| = s+1$ and $|I_i| = N-N_i$. Define $\bar{D}_i = D_i \cup I \cup I_i$ and note $|\bar{D}_i| = N_i+N-N_i+s+1 = N+s+1$ for all $i$. Finally, define $f\colon D \to \mathcal{P}_p(\mathbb{R}^2)$ by sending each $D_i$ to the uniform measure on $\bar{D_i}$. We first check that $f$ has the desired upper bound. Let $\sigma_{ij}: \mathcal{U}_{ij} \to \sigma_{ij}(\mathcal{U}_{ij})$ be an optimal matching from $\mathcal{U}_{ij}$ to a subset of $I$
and $\tau_{ji}$ denote an optimal matching from $\mathcal{U}_{ji}$ to a (possibly different) subset of $I$. Now, note that,
\begin{align*}
|I\cup I_i \setminus \tau_{ji}(\mathcal{U}_{ji})| &=s+1+N-N_i-|\mathcal{U}_{ji}|\\
&= s+1+N-N_j-|\mathcal{U}_{ij}|\\
&= |I\cup I_j\setminus \sigma_{ij}(\mathcal{U}_{ij})|.
\end{align*} 
 Assume that the sets $I\cup I_i \setminus \tau_{ji}(\mathcal{U}_{ji})$ and $I\cup I_j\setminus \sigma_{ij}(\mathcal{U}_{ij})$ have been ordered as $\{x_i\}$ and $\{y_i\}$ respectively, with ascending coordinates, and let $\omega$ denote the bijection which sends $x_i$ to $y_i$. We define a coupling, $\pi$, between the uniform measures $f(D_i)$ and $f(D_j)$ by, 
 \begin{align*}
    \pi = \frac{1}{N+s+1}\left[ \sum_{x \in D_i\setminus \mathcal{U}_{ij}}\delta_{(x,\pi_{ij}(x))} + \sum_{x \in \mathcal{U}_{ij}}\delta_{(x,\sigma_{ij}(x))} +\sum_{x \in \tau_{ji}(\mathcal{U}_{ji})} \delta_{(x,\tau_{ji}^{-1}(x))} + \sum_{x \in I\cup I_i \setminus \tau_{ji}(\mathcal{U}_{ji})} \delta_{(x,\omega(x))}    \right].
 \end{align*}

With $x = \left(m+\frac{t_1(M-m)}{s},m+\frac{t_1(M-m)}{s}\right)$ and $\omega(x) = \left(m+\frac{t_2(M-m)}{s},m+\frac{t_2(M-m)}{s}\right)$ we note that, 
\begin{align*}
\|x-\omega(x)\|_{\infty}
&= \frac{|t_1-t_2|(M-m)}{s}\\
&\leq \frac{N(M-m)}{s}.
\end{align*}

Thus, 
\begin{align*}
    \sum_{x \in I\cup I_i\setminus \tau_{ji}(\mathcal{U}_{ij})} \|x-\omega(x)\|_{\infty}^p \leq  (s+1+N-N_j-|\mathcal{U}_{ij}|)\left( \frac{N(M-m)}{s} \right)^p < \frac{\epsilon^p}{3}.
\end{align*}

Since $\pi$ is a coupling between $f(D_i)$ and $f(D_j)$,

\begin{align*}
    & d(f(D_i),f(D_j)) & \\
    \leq & \left( \int_{\mathbb{R}^2\times \mathbb{R}^2}\|x-y\|_{\infty}^pd\pi(x,y)\right)^{\frac{1}{p}}\\
 =&\frac{1}{(N+s+1)^{1/p}}\left[ \sum_{x \in D_i\setminus   
    \mathcal{U}_{ij}}\|x-\pi_{ij}(x)\|_{\infty}^p + 
    \sum_{x \in \mathcal{U}_{ij}}\|x -\sigma_{ij}(x)\|_{\infty}^p +\right.\\
        &\sum_{x \in \tau_{ji}(\mathcal{U}_{ji})} \|x -\tau_{ji}^{-1}(x)\|_{\infty}^p + \sum_{x \in I\cup I_i \setminus \tau_{ji}(\mathcal{U}_{ji})} \|x - \omega(x)\|_{\infty}^p    \left.\vphantom{\sum_{x \in D_i\setminus   
    \mathcal{U}_{ij}}}\right]^{\frac{1}{p}}\\
\leq&\frac{1}{(N+s+1)^{1/p}}\left[ \sum_{x \in D_i\setminus   
    \mathcal{U}_{ij}}\|x-\pi_{ij}(x)\|_{\infty}^p + 
    \sum_{x \in \mathcal{U}_{ij}}\left[\|x-\rho_x\|_{\infty}+\|\rho_x-\sigma_{ij}(x)\|_{\infty}\right]^p +\right.\\
        &\sum_{x \in \tau_{ji}(\mathcal{U}_{ji})}\left[\|x-\rho_{\tau_{ji}^{-1}(x)}\|_{\infty}+\|\rho_{\tau_{ji}^{-1}(x)}-\tau_{ji}^{-1}(x)\|_{\infty}\right]^p + \sum_{x \in I\cup I_i \setminus \tau_{ji}(\mathcal{U}_{ji})} \|x - \omega(x)\|_{\infty}^p    \left.\vphantom{\sum_{x \in D_i\setminus   
    \mathcal{U}_{ij}}}\right]^{\frac{1}{p}}\\
\leq&\frac{1}{(N+s+1)^{1/p}}\left[ \sum_{x \in D_i\setminus   
    \mathcal{U}_{ij}}\|x-\pi_{ij}(x)\|_{\infty}^p + \sum_{x \in \mathcal{U}_{ij}}\|x-\rho_x\|_{\infty}^p+\sum_{x \in \tau_{ji}^{-1}(\mathcal{U}_{ji})} \|x-\rho_x\|_{\infty}^p\right]^{\frac{1}{p}} +\\
    &\frac{1}{(N+s+1)^{1/p}}\left[\sum_{x \in \mathcal{U}_{ij}}\|\rho_x-\sigma_{ij}(x)\|_{\infty}^p+\sum_{x \in \mathcal{U}_{ji}} \|\rho_x-\tau_{ji}(x)\|_{\infty}^p+ \sum_{x \in I\cup I_i \setminus \mathcal{U}_{ij}} \|x - \omega(x)\|_{\infty}^p \right]^{\frac{1}{p}}\\
    \leq & \frac{1}{(N+s+1)^{1/p}}\left[d(D_i,D_j) + \epsilon\right].
\end{align*}
where the last step is obtained by noting that $\rho_x$ and $\sigma_{ij}(x)$ are at most $\frac{N(M-m)}{s}$ apart, and applying our choice of $\epsilon$. This completes the proof of the upper bound. 

\redd{The proof of the lower bound is the same as in the proof of Lemma \ref{PDembedbi}}.

\end{proof}

\begin{figure}[h]
    \centering
    \begin{tikzpicture}[scale=2]
    \begin{scope}
        \draw[->] (0,0)--(0,2);
        \draw[->] (0,0)--(2,0);
        \draw[dashed] (0,0)--(2,2);
        \node[anchor=center] at (1,1.5) {\color{blue} $\blacktriangle$};
        \node[anchor=center] at (0.2,1.4) {\color{blue} $\blacktriangle$};
        \node[anchor=center] at (0.3,0.7) {\color{orange} $\bullet$};
        \node[anchor=center] at (0.5,1.9) {\color{orange} $\bullet$};
        \node[anchor=center] at (1.5,1.9) {\color{orange} $\bullet$};
        \draw[|->] (2.5,1)--(3,1); 
    \end{scope}
    \begin{scope}[xshift=100]
        \draw[->] (0,0)--(0,2);
        \draw[->] (0,0)--(2,0);
        \node[anchor=center] at (1,1.5) {\color{blue} $\blacktriangle$};
        \node[anchor=center] at (0.2,1.4) {\color{blue} $\blacktriangle$};
        \node[anchor=center] at (0.3,0.7) {\color{orange} $\bullet$};
        \node[anchor=center] at (0.5,1.9) {\color{orange} $\bullet$};
        \node[anchor=center] at (1.5,1.9) {\color{orange} $\bullet$};
        \foreach \i in {5,...,18}{
            \node at (\i/10, \i/10) {\color{blue} $\blacktriangle$};
        }
        \foreach \i in {5,...,17}{
            \node at (\i/10, \i/10) {\color{orange} $\bullet$};
        }
    \end{scope}
    \end{tikzpicture}
    \caption{Sketch of the construction in the proof Lemma \ref{PDembedbi} for two diagrams $\{D_1, D_2\}$. The two persistence diagrams on the left (blue triangles and orange circles resp.) are sent by $f$ to the uniform distributions shown on the right, whose supports have the same cardinality. }
    \label{fig:proof}
\end{figure}
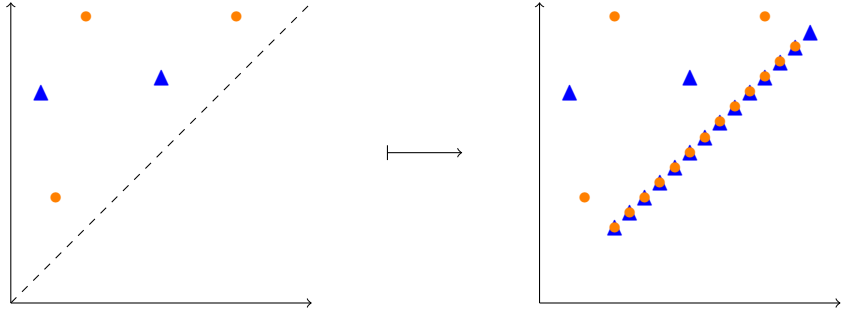

For any persistence diagram there exists a diagram whose points have multiplicity one and which is arbitrarily close to the original diagram. Thus, one may relax the restriction that each diagram has multiplicity one. This, together with dilating measures using $\lambda_{(N+s+1)^{1/p}}$, results in the following form of Lemma \ref{PDembed} which will be more useful. 

\begin{proposition}\label{eqipd}
Suppose $D = \{D_1,D_2,\cdots, D_n\} \subset (\mathcal{D},W_p)$ is a finite subset of diagrams and let $p>1$. Then for all $\epsilon >0$ there exists an $\epsilon$-quasi-isometric embedding $f\colon D \to \mathcal{P}_p(\mathbb{R}^2)$.
\end{proposition}

\section{\redd{Main results}}

 \redd{We are now ready to prove our main results concerning coarse embeddability and snowflake universality of persistence diagrams and Wasserstein space. Together, Proposition~\ref{equipdbi}, and Theorem~\ref{nowak} give the following result, which connects the embeddability questions for persistence diagrams and Wasserstein space.

\begin{theorem}\label{A}
For any $p \in [1,\infty)$, $(\mathcal{D},W_p)$ coarsely embeds into Hilbert space if and only if $\mathcal{P}_p(\mathbb{R}^2)$ coarsely embeds into Hilbert space. 
\end{theorem}
\begin{proof}
Suppose $(\mathcal{D},W_p)$ coarsely embeds into Hilbert space. By Proposition~\ref{equipdbi}, every finite subset of $\mathcal{P}_p(\mathbb{R}^2)$  coarsely embeds (uniformly) into $(\mathcal{D},W_p)$, so by composing embeddings we obtain that every finite subset of $\mathcal{P}_p(\mathbb{R}^2)$ coarsely embeds (uniformly) into Hilbert space. By Theorem~\ref{nowak}, it follows that $\mathcal{P}_p(\mathbb{R}^2)$ itself coarsely embeds into Hilbert space. The converse is similar.
\end{proof}
}

The following corollary follows from Theorems \ref{waginf}, \ref{wagp}, and \ref{A}. We note that a direct proof is also possible adapting techniques from Wagner~\cite{wagner2021nonembeddability}.

\begin{corollary}
    The space $\mathcal{P}_p(\mathbb{R}^2)$ does not admit a coarse embedding into a Hilbert space if $p>2$.
\end{corollary}

Recall that Andoni-Naor-Neiman showed that $\mathcal{P}_p(\mathbb{R}^3)$ is $\frac{1}{p}$-snowflake universal for $p > 1$, from which it follows that $\mathcal{P}_p(\mathbb{R}^3)$ does not admit a coarse embedding into any Banach space of non-trivial type for $p > 1$~\cite{andoni2018snowflake}. Whether this holds for $\mathcal{P}_p(\mathbb{R}^2)$, however, is still open. In addition to coarse embeddabilty, we obtain a result which connects the questions of snowflake universality for $\mathcal{P}_p(\mathbb{R}^2)$ and $(\mathcal{D}, W_p)$.  

\begin{lemma}\label{snow}
Let $X$ and $Y$ be metric spaces and suppose for all $\epsilon>0$ and all finite subsets $A \subset X$ there exists a map $f\colon A \to Y$ satisfying
\begin{align}\label{eqi1}
    d(x_i,x_j)-\epsilon \leq d(f(x_i),f(x_j)) \leq d(x_i,x_j)+\epsilon
\end{align}
for all $x_i,x_j \in A$. Then $X$ being $\theta$-snowflake universal implies $Y$ is $\theta$-snowflake universal. 
\begin{proof}
Let $W$ be the $\theta$-snowflake of some finite metric space. Let $D = 1+\epsilon > 1$ be given
and take $\delta<1$ so small so that $\frac{1+\delta}{1-\delta}<1+\epsilon$. Take $\epsilon_1<\frac{\delta}{2}$. Then by assumption there exists a map $g\colon W \to X$ and a constant $k_1$ such that,
\begin{align*}
    k_1d(x_i,x_j)\leq d(g(x_i),g(x_j))\leq k_1(1+\epsilon_1)d(x_i,x_j).
\end{align*}
Let $M = \min \{d(x_i,x_j) \mid x_i, x_j \in W,\ x_i \neq x_j\}$ and take $\epsilon_2 < \frac{\delta k_1M}{2}$. Let $f\colon g(W) \to Y$ be a map satisfying \eqref{eqi1} with respect to $\epsilon_2$. We claim that $f\circ g: W \to Y$ has distortion $D$. Indeed,
\begin{align*}
    d(f(g(x_i)),f(g(x_j))) &\leq d(g(x_i),g(x_j))+\epsilon_2\\
    &\leq k_1(1+\epsilon_1)d(x_i,x_j)+\epsilon_2\\
    &\leq k_1(1+\delta)d(x_i,x_j)\\
    &\leq k_1(1-\delta)(1+\epsilon)d(x_i,x_j).
\end{align*}
Further, 
\begin{align*}
    d(f(g(x_i)),f(g(x_j))) &\geq d(g(x_i),g(x_j))-\epsilon_2\\
    &\geq k_1d(x_i,x_j)-\epsilon_2\\
    &= k_1d(x_i.x_j)\left[1-\frac{\epsilon_2}{k_1d(x_i,x_j)}\right]\\
    &\geq k_1(1-\delta)d(x_i,x_j)
\end{align*}
\end{proof}
\end{lemma}

Lemma \ref{snow} with Propositions \ref{eqi} and \ref{eqipd} yield the following.

\begin{theorem}\label{B}
If $\mathcal{P}_p(\mathbb{R}^2)$ is $\frac{1}{p}$-snowflake universal, for $1 \leq  p < \infty$, then so is $(\mathcal{D},W_p)$. The converse holds if $p  > 1$.
\end{theorem}

\section{Concluding remarks}

Wasserstein space and the space of persistence diagrams have many similarities, especially when viewing persistence diagrams as discrete distributions; this similarity has been explored elsewhere using partial optimal transport~\cite{divol2021understanding}. The main difference between the spaces is the presence of the diagonal as a sink for unmatched points. Our results suggest that this difference does not affect the coarse embeddability of these spaces. \redd{Obstructions to coarse embeddability into Hilbert space developed in either case will therefore work just as well for the other.  Snowflake universality is stronger than coarse non-embeddability into Hilbert space, but for} $p=1$, our construction degenerates in a similar way to that in \cite{andoni2018snowflake} and so we do not obtain an equivalence. The $p=1$ case is important since it appears in many stability results for vectorizations~\cite{adams2017persistence,chung2022persistence}. Note that stability is one half of coarse (or bi-Lipschitz) embeddability, as it bounds the distortion in Hilbert space in terms of the distance between diagrams. Our hope is that our results motivate the use of techniques from optimal transport (such as those in \cite{andoni2018snowflake}) to resolve the still open question of coarse embeddability of persistence diagrams for $1 \leq p \leq 2$.

\section*{Statements and Declarations}
\textbf{Competing interests}. The authors have no competing interests to declare that are relevant to the content of this article. 

\textbf{Data availability}. Data sharing not applicable to this article as no datasets were generated or analysed during the current study.

\bibliographystyle{plain}
\bibliography{references}

\end{document}